\documentclass[a4paper,12pt]{article}
\usepackage{setspace}
\usepackage[normalem]{ulem}
\usepackage{latexsym}

\usepackage[utf8]{inputenc}
\usepackage[english]{babel}

\usepackage{amsfonts,amsmath,amssymb,amsbsy,amsthm,enumerate,lineno,lmodern,microtype}
\usepackage[colorlinks,
  linkcolor=red!60!black,%
  citecolor=green!40!black,%
  urlcolor=black,]{hyperref}

\usepackage{eucal}

\usepackage{pstricks}
\usepackage{multirow}
\usepackage{verbatim}
\usepackage{color}
\usepackage{enumitem}

\newcommand{\authormark}[1]{\textsuperscript{\,#1}}
\newcommand{\showmark}[1]{
  \hspace*{-1em}\makebox[1em][r]{\authormark{#1}\,}
  \ignorespaces}

\newcommand{\defi}[1]{\emph{#1}}
\DeclareMathOperator{\bigoh}{\mathrm{O}}

\newtheorem{theorem}{Theorem}

\newtheorem{conjecture}[theorem]{Conjecture}

\newtheorem{lemma}[theorem]{Lemma}

\makeatletter
\renewcommand{\@biblabel}[1]{\textcolor{gray}{[\,}#1\textcolor{gray}{\,]}}
  \renewcommand\@openbib@code{
    \setlength\labelwidth{2cm}
    \setlength{\itemindent}{0cm}
    \setlength{\leftmargin}{0cm}
    \setlength\labelsep{.8em}
    }
\makeatother

\usepackage{tikz}
\usepackage{xifthen}
\usepackage{subfig}
\usetikzlibrary{calc,positioning,decorations.pathmorphing,decorations.pathreplacing}
\tikzset{emp/.style={double distance = 0.3ex}}
\tikzset{oriented/.style={->,shorten >= 1.5pt}}

\usetikzlibrary{calc,decorations.markings,decorations.pathmorphing,arrows,external,decorations.pathreplacing,arrows.meta,calc,decorations.markings,decorations.pathmorphing,shapes}

\tikzset{highlight/.style={orange!75!white,line width=6pt,line cap=round}}

\usepackage{graphicx}
\usepackage{enumitem}
\setlist[enumerate]{label=\textit{(\roman*)},ref=\textit{(\roman*)}}
\title{Separating the edges of a graph \\ by cycles and by subdivisions of~\(K_4\)}
\author{F\'abio Botler\authormark{1}
\and T\'assio Naia\authormark{2}}

\newcommand{\calk}{\mathcal{K}}

\newcommand{\cals}{\mathcal{S}}

\begin{document}

\setlength\linenumbersep{1.7cm}

\maketitle

\begin{center}
\footnotesize
\begin{minipage}{.45\textwidth}
\begin{center}
\showmark{1}
Departamento de Ci\^encia da Computa\c c\~ao \\
Instituto de Matem\'atica e Estat\'\i stica \\
Universidade de S\~ao Paulo, Brasil\\
E-mail: \texttt{fbotler@ime.usp.br}
\end{center}
\end{minipage}
\begin{minipage}{.45\textwidth}
\begin{center}
\showmark{2}
{Centre de Recerca Matem\`atica\\
  Belaterra, Spain\\}
E-mail: \texttt{tnaia@member.fsf.org}
\end{center}
\end{minipage}
\end{center}

\newcommand{\RR}{82}
\newcommand{\tncomm}[1]{{\color{red}#1}}
\newcommand{\fbcomm}[1]{{\color{blue}#1}}

\begin{abstract}
A \emph{separating system} of a graph~\(G\)
is a family \(\mathcal{S}\) of subgraphs of \(G\) for which the following holds:
for all distinct edges $e$ and~$f$ of~$G$, there exists an element in $\mathcal{S}$
that contains \(e\) but not~\(f\).
Recently, it has been shown that every graph of order~\(n\)
  admits a separating system consisting of \(19n\) paths
  [Bonamy, Botler, Dross, Naia, Skokan, \textsl{Separating the Edges of a Graph by a Linear
Number of Paths}, Adv. Comb., October 2023],
improving the previous almost linear bound of \(\bigoh(n\log^\star n)\)
[S.~Letzter, \textsl{Separating paths systems of almost linear size}, Trans.\ Amer.\ Math.\ Soc., to appear],
and settling conjectures posed
by~Balogh, Csaba, Martin, and Pluh\'ar and
by~Falgas-Ravry, Kittipassorn, Kor\'andi, Letzter, and Narayanan.

We investigate a natural generalization of these results to subdivisions of cliques,
showing that every graph admits both a
separating system consisting of \(41n\) edges and cycles,
and a separating system consisting of \(\RR n\) edges and subdivisions of~\(K_4\).
\end{abstract}

Consider a family \(\cals\) of subsets of a set \(X\).
Given \(e,f\in X\), we say that an element \(S\in\cals\) \emph{separates} \(e\) from \(f\)
if \(e\in S\) and \(f\notin S\).
Furthermore, we say that \(\cals\) \emph{weakly separates} \(X\)
if for every pair of elements \(e,f \in X\) there is an element \(S\in\cals\)
that either separates \(e\) from \(f\) or that separates \(f\) from \(e\);
and we say that \(\cals\) \emph{strongly separates}~\(X\)
if for every pair of elements \(e,f \in X\) there is an element \(S\in\cals\)
that separates \(e\) from \(f\), and an element that separates \(f\) from~\(e\).
Such a family \(\cals\) is then called a (weak or strong, respectively) \emph{separating system} of \(X\).
The study of separating systems dates back to the 1960s~\cite{KATONA1966174,renyi1961random,spencer1970minimal},
and gained substantial attention in the last years (see, e.g.,~\cite{arrepol2023separating,biniaz2023separating,BOLLOBAS20071068}).
Our focus is a variation posed in 2013 by Katona~(see~\cite{balogh2016path,falgas2013separating})
in which one seeks to separate the \emph{edge} set of a given graph by a small collection of~paths.
In particular, special attention has been given to the case of complete graphs~\cite{FernandesMotaSanhuezaMatamala24,kontogeorgiou2024exact,WICKES2024113784}.

For our purposes, a \defi{separating path system} (resp.\ \defi{separating cycle system})
of a graph~$G$ is a collection~$\cals$ of paths (resp.~edges and~cycles) in~$G$
such that for all distinct edges $(e,f)\in E(G)\times E(G)$ there exists $S\in\cals$
that separates \(e\) from~\(f\).
Note that this definition fits into the category of \emph{strong} separating systems.
Also note that, in the definition of cycle separating system, one is allowed to use isolated edges.
This is required in order to separate graphs that contain bridges (i.e., edges that do not lie in any cycle).

Inspired by Katona's question,
Falgas-Ravry, Kittipassorn, Kor\'andi, Letzter, and Narayanan~\cite{falgas2013separating}
conjectured that every graph on $n$~vertices admits a weak separating system size~$\bigoh(n)$
consisting solely of paths, while verifying it for a number of cases.
Their conjecture was later strengthened
by~Balogh, Csaba, Martin, and Pluh\'ar~\cite{balogh2016path}.

\begin{conjecture}[\cite{balogh2016path,falgas2013separating}]\label{cj:path-sep}
  Every graph of order $n$ admits a path separating system of size~$\bigoh(n)$.
\end{conjecture}

In 2023, Bonamy, Dross, Skokan and the authors confirmed Conjecture~\ref{cj:path-sep},
proving that every graph on \(n\)~vertices admits a separating path system of size at most
\(19n\)~\cite{BonamyBotlerDrossNaiaSkokan23}.
This improved the almost linear bound of \(\bigoh(n\log^\star n)\) found by Letzter
in 2022~\cite{letzter2022separating}.
It is then natural to ask whether
every graph also admits a separating cycle system of size~$\bigoh(n)$.
This question was independently posed by Girão and Pavez-Sign{\'e}\footnote{Personal communication.},
and in this paper we answer it affirmatively (see Section~\ref{sec:cycles}).

\begin{theorem}\label{thm:cycles}
  Every graph on~$n$ vertices admits a separating cycle system of size~\(41n\).
\end{theorem}

In fact, we are interested in a more general setting in which
elements of the separating system are either edges or subdivisions of a given graph,
which we now make precise.
We say that \(H^*\) is a \emph{subdivision} of \(H\) if $H$ can be obtained from~$H^*$
by repeatedly deleting a vertex of degree~\(2\) and adding a new edge joining its neighbors.
Let \(H\) and \(G\) be graphs, and let \(\mathcal{F}\) be a family of subgraphs of \(G\)
such that every element of \(\mathcal{F}\) is either an edge or a subdivision of \(H\).
We say that \(\mathcal{F}\) is an \emph{\(H\)-cover} of~\(G\) if \(E(G) = \bigcup_{H'\in\mathcal{F}} E(H')\);
and we say that \(\mathcal{F}\) is an \emph{\(H\)-separating system} of \(G\)
if for all distinct edges \((e,f)\in E(G)\times E(G)\) there is an element \(H'\in \mathcal{F}\)
that separates \(e\) from \(f\).
It is not hard to see that every \(H\)-separating system of \(G\) is an \(H\)-cover of~\(G\),
and that a separating path (resp.~cycle) system is a $K_2$-separating (resp.~$K_3$-separating) system.
Therefore,
the results mentioned above say that every graph on \(n\)~vertices
admits a \(K_2\)- and a \(K_3\)-separating system of size~\(O(n)\).
Motivated by Conjecture~\ref{cj:path-sep}, we propose the following more general
edge separation conjecture.

\begin{conjecture}\label{conj:edge-separation}
	For every graph \(H\) there is a constant \(C = C(H)\) for which every graph on \(n\) vertices
	admits an \(H\)-separating system of size at most \(C \cdot n\).
\end{conjecture}

It is not hard to check that in order to solve Conjecture~\ref{conj:edge-separation},
it suffices to verify it in the case \(H\) is a complete graph.
The main result of this paper is that Conjecture~\ref{conj:edge-separation}
holds in the case \(H = K_4\) (see Section~\ref{sec:main}).

\begin{theorem}\label{thm:K4-separation}
  Every graph on~$n$ vertices admits a \(K_4\)-separating cycle system of size~\(\RR n\).
\end{theorem}

Since every subdivision of \(K_4\) can be covered by two
cycles, Theorem~\ref{thm:K4-separation} implies a version of
Theorem~\ref{thm:cycles} in which the constant \(41\) is
replaced by \(164\).

The strategy presented here is similar to the strategy in~\cite{BonamyBotlerDrossNaiaSkokan23}.
Namely, we reduce the main problem to the case of graphs containing a certain spanning subdivision of a~clique.
Next, we define a linear number of special matchings that separate the edges outside
this structure, and cover each such matching by a suitable subdivision.
We emphasize that we make no effort to reduce the multiplicative constant.

\medskip
\noindent
\textbf{Path separation versus $K_4$ separation.}
In~\cite{BonamyBotlerDrossNaiaSkokan23}, the authors reduce the problem to the case of
graphs containing a Hamilton path~$P=v_1\cdots v_n$.
They consider $5n$ matchings~\(M_k = \{v_iv_j \in E(G): i+j = k,\, i<j\}\) and \(N_k = \{v_iv_j \in E(G): i+2j = k,\,i<j\}\)
and show that each such matching~$M$ can be covered by a path~$P_M$ with $M\subseteq E(P_M)\subseteq E(P\cup M)$.
The argument is completed using a linear path-covering result.

In contrast, here the reduction is to graphs containing a subdivision~$K$ of $K_4$ with
the property that $K$ contains a Hamilton cycle~$C=v_1\cdots v_nv_1$.
We consider $3n$ matchings~\(M_k = \{v_iv_j \in E(G): j-i = k,\, i<j\}\) and \(N_k = \{v_iv_j \in E(G) : j-2i = k,\,i<j\}\).
While some of these matchings cannot be covered by a single subdivision (see Figure~\ref{fig:bad-example}),
we can show that a bounded number of subdivisions suffice.
The argument is completed using a linear clique-subdivision-covering result.

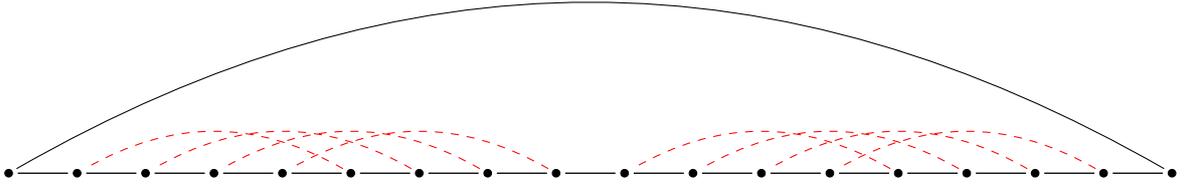
\begin{figure}[h]
\centering
\begin{tikzpicture}[node distance=2cm,scale = .9]
\tikzset{black vertex/.style={circle,draw,minimum size=1mm,inner sep=0pt,outer sep=2pt,fill=black, color=black}}

\foreach \x in {1,...,18}{
    \node[black vertex] (\x) at (\x,0) {};
}

\draw[line width = .5pt] (1) -- (2) -- (3) -- (4) -- (5) -- (6) -- (7) -- (8) -- (9) -- (10) -- (11) -- (12) -- (13) -- (14) -- (15) -- (16) -- (17) -- (18);

\foreach \i in {6,7,8,9,14,15,16,17}
{
    \pgfmathtruncatemacro{\j}{(\i-4)}
    \ifnum\j>0
        \draw[bend right,color=red,dashed] (\i) to (\j);
    \fi
}
\draw[bend right = 30] (18) to (1);

\end{tikzpicture}
\caption{A graph consisting of a cycle together with a matching \(M \subseteq \{u_iu_j: j-i = 4\}\) (in dashed red) having no cycle that contains all edges of~\(M\).}
\label{fig:bad-example}
\end{figure}

\bigskip
\noindent
\textbf{P\'osa rotation-extension.}
We use the following standard notation.
Given a graph \(G\), and  a set \(S\subseteq V(G)\),
we denote by \(N_G(S)\) the set of vertices \emph{not} in $S$ adjacent in \(G\) to some vertex in \(S\).
We omit subscripts when clear from the context.

Given a graph $G$ and vertices $u, v$ in \(G\), let \(P = u \cdots v\) be a path from $u$ to $v$.
If \(x\in V(P)\) is a neighbor of \(u\) in \(G\)
and \(x^-\) is the vertex preceding \(x\) in~\(P\),
then \(P' = P-xx^- + ux\) is a path in \(G\) for which \(V(P') = V(P)\).
We say that \(P'\) \defi{has been obtained from} \(P\) by an
\defi{elementary exchange} fixing~\(v\) (see Figure~\ref{fig:rotation}).
A path obtained from~\(P\) by a (possibly empty) sequence of elementary exchanges fixing~\(v\)
is said to be a path \defi{derived} from~\(P\).
The set of endvertices of paths derived from~\(P\) distinct from~\(v\) is denoted by~\(S_v(P)\).
Since all paths derived from~\(P\) have the same vertex set as~\(P\), we have~\(S_v(P) \subseteq V(P)\).
The following lemma arises when \(P\) is a longest path ending at \(v\) (for a proof see also~\cite{BonamyBotlerDrossNaiaSkokan23}).

\begin{figure}[ht]
  \begin{center}
    \tikzset{highlight/.style={orange!75!white,line width=6pt,line cap=round}}
\tikzset{highlight2/.style={tngreen!70!yellow!60,line width=11pt,line cap=round}}
\tikzset{psline/.style={decorate,decoration={amplitude=0.3mm,segment length=2mm,post length=1.5mm,pre length=1.5mm,coil,aspect=0}}}
\tikzset{plainlin/.style={line width=1pt}}
\def\shortendistance{5pt}
\tikzset{plainlin shorten/.style={line width=1pt,shorten >=\shortendistance,shorten <=\shortendistance}}

\tikzset{lin shorten/.style={shorten >=\shortendistance,shorten <=\shortendistance,line width=1pt,decoration={markings, mark=at position 0.5*\pgfdecoratedpathlength+1.8pt with {\arrow{angle 90}}}, postaction={decorate}}}

\newcommand{\vertexat}[1]{\fill [very thick,draw=white,fill=black] #1 circle (3pt);}

\begin{tikzpicture}

    \coordinate  (u)   at  (0,  1.5);
    \coordinate  (x-)  at  (2,  1.5);
    \coordinate  (x)   at  (3,  1.5);
    \coordinate  (v)   at  (5,  1.5);

    \begin{scope}[opacity=.4]
      \draw [highlight,rounded corners] (x-) to (u) to [out= 45, in=135] (x)
      to  (v);
    \end{scope}

    \draw [thick] (u) to (x-) to (x) to (v);
    \draw [thick] (u) to [out=45,in=135] (x);

    \foreach \nn in {u,x-,x,v} {
      \vertexat{(\nn)};
    }

    \node [anchor=north] at (u) {$u\vphantom{x^-}$};
    \node [anchor=north] at (x-) {$x^-$};
    \node [anchor=north] at (x) {$x\vphantom{x^-}$};
    \node [anchor=north] at (v) {$v\vphantom{x^-}$};
\end{tikzpicture}
    \caption{a path (highlighted)
      obtained by an elementary exchange fixing~\(v\).}
    \label{fig:rotation}
  \end{center}
\end{figure}
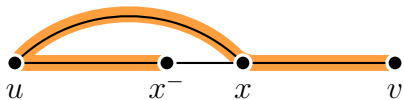

\begin{lemma}[\cite{brandt2006global}]\label{lemma:brandt}
  Let  \(P = u \cdots v\) be a longest path of a graph $G$
  and let \(S = S_v(P)\). Then \(|N_G(S)| \leq 2|S|\).
\end{lemma}

\section{$K_4$-separating systems}\label{sec:main}

In this section we verify Conjecture~\ref{conj:edge-separation} when \(H = K_4\).
Our argument requires a result stating that every $n$-vertex graph admits a \(K_4\)-cover
of size $\bigoh(n)$.
Although we can prove such a statement simultaneously to the separation result,
we obtain a better leading constant by using
the following result due to J{\o}rgensen and Pyber~\cite{jorgensen1990covering}.
In fact, J{\o}rgensen and Pyber showed that every $n$-vertex graph admits a $K_t$-cover
of size $\bigoh_t(n)$.

\begin{theorem}[{\cite{jorgensen1990covering}}]\label{thm:covering-main}
Every graph on \(n\) vertices admits a \(K_4\)-cover of size at most \(2n-3\).
\end{theorem}

The proof of Theorem~\ref{thm:K4-separation} is divided into four parts.
First we apply Lemma~\ref{lemma:brandt} in order to reduce the problem to the case where the studied graph contains a special spanning subdivision of \(K_4\);
then we define the \(M_k\)'s and \(N_k\)'s and show how to partition them into matchings in order to avoid the
problem illustrated in Figure~\ref{fig:bad-example};
the third step is to consider the union of each matching and the spanning subdivision of \(K_4\),
and show that it contains another subdivision of~\(K_4\) covering the matching;
and we finally argue that the obtained collection is the desired \(K_4\)-separating system.

Let $H_1$ and~$H_2$ be (not necessarily edge-disjoint) subgraphs of a graph~$G$,
and consider a family $\mathcal{S}$ of subgraphs of~$G$.
We say that $\mathcal{S}$ \defi{separates $H_1$ from $H_2$}
if for all distinct edges \((e,f) \in E(H_1) \times E(H_2)\),
there is $S\in\mathcal{S}$ that separates \(e\) from \(f\).
Therefore, \(\mathcal{S}\) is a separating system of \(G\) if \(\mathcal{S}\) separates \(G\) from itself.

\begin{proof}[Proof of Theorem~\ref{thm:K4-separation}]
We proceed by induction on $n$.
Let \(G\) be a graph with \(n\) vertices.
If~\(G\)~is empty, the result trivially holds.
If not, we consider a longest path $P=u\cdots v$ of~$G$ and
put $S=S_v(P)$ as in Lemma~\ref{lemma:brandt}.
Now, let \(w\) be the vertex that is closest to $v$ in~\(P\) and has a neighbor~$u'$ in~\(S\),
and let \(P'=u'\cdots v\) be a path obtained from~\(P\) by elementary exchanges fixing~\(v\)
so that \(P'\) starts with~\(u'\).
Then \(C = (P' + u'w) \setminus E(P[w,v])\) is an edge or a cycle that contains \(S\cup N(S)\)
(\(C\) is an edge when \(S = \{u\}\) and \(|N(S)| = 1\)).

Let~\(H\) be the subgraph of \(G\) induced by the edges incident to vertices of \(S\),
and note that \(h = |V(H)| \leq 3|S|\).
Let \(G' = G \setminus S\).
Note that, since \(G\) is not empty, \(S\) is not empty.
By the induction hypothesis,
there is a \(K_4\)-separating system \(\mathcal{S}'\) of \(G'\) of size at most \(\RR(n-|S|)\).
Note that \(\mathcal{S}'\) separates  \(G'\) from \(H\) since \(\mathcal{S}'\) covers~\(G'\).
In what follows,
we construct a set \(\mathcal{S}\) of edges and subdivisions of \(K_4\) that separates
\(H\) from \(G\).
Moreover, we obtain \(\mathcal{S}\) so that \(|\mathcal{S}| \leq \RR|S|\),
and hence \(\mathcal{S}' \cup \mathcal{S}\)
is a \(K_4\)-separating system of \(G\)
with cardinality at most  \(\RR(n-|S|) + \RR |S| = \RR n\) as desired.

Now, we show that either \(e(H)\le \RR|S|\)
or $H$~contains a subdivision of \(K_4\) that contains~\(C\).
For that, let us write \(C = u_1 \cdots u_mu_1\).
We remark that $C$ may contain vertices in $V(G)\setminus V(H)$.
Let $v_1,\ldots,v_h$ be the vertices of~$H$
in the order that they appear in $C$.
Formally, for $i \in [h]=\{1,\ldots,h\}$,
we define $\sigma(i)$ to be the index of the $i^{\textrm{th}}$ vertex $u_{\sigma(i)}$ of \(H\) in~$C$,
and set $v_i=u_{\sigma(i)}$.
In what follows, each edge \(v_iv_j\) is written so that \(i < j\).
We say that two edges \(v_iv_j\) and \(v_{i'}v_{j'}\) of \(H\setminus E(C)\) \emph{cross}
if $v_{i'}$ and $v_{j'}$ lie in distinct components of $C\setminus \{v_i,v_j\}$.
If no two edges in~\(H\setminus E(C)\) cross,
then \(H\) is an outerplanar graph,
and hence \(e(H)\le 2h-3\) (note that $h\ge 2$).
In this case we set \(\mathcal{S}\) as
the set of subgraphs each consisting of a single edge in~\(H\),
and we are done because \(|\mathcal{S}|=|e(H)| \leq 2h-3 \leq 6|S| \leq \RR |S|\).

Therefore, we may assume that there are crossing edges \(e,e'\) in \(H\setminus E(C)\).
Note that in this case \(K = C + e + e'\) is a subdivision of \(K_4\).
Let \(K_S = E(K) \cap E(H)\) be the set of edges of \(K\)
having at least one vertex in \(S\) (see Figure~\ref{fig:NS}),
and let \(\calk_S\) be
the set of subgraphs each consisting of a single edge in \(K_S\).
Observe that $|\calk_S|=e(K_S)\le 2|S|+2$.
By Theorem~\ref{thm:covering-main}, there is a \(K_4\)-cover \(D'\) of \(H' = H\setminus E(K_S)\)
of size at most \(2 h -3\leq 6 |S|\).
Note that \(\mathcal{K}_S \cup D'\)
has size at most \(2|S| + 2 + 6|S| \leq 10|S|\) (using that \(2 \leq 2|S|\)) and separates
(i)  \(H\) from \(G'\);
(ii)  \(H'\) from \(K\);
and (iii) \(K_S\) from \(G\).

\begin{figure}
  \vspace*{-1.5cm}
  \begin{center}
    \hfill
    \begin{minipage}{.48\textwidth}
      \tikzset{highlight/.style={orange!75!white,line width=6pt,line cap=round}}
\tikzset{highlight2/.style={tngreen!70!yellow!60,line width=11pt,line cap=round}}
\tikzset{psline/.style={decorate,decoration={amplitude=0.3mm,segment length=2mm,post length=1.5mm,pre length=1.5mm,coil,aspect=0}}}
\tikzset{plainlin/.style={line width=1pt}}
\def\shortendistance{5pt}
\tikzset{plainlin shorten/.style={line width=1pt,shorten >=\shortendistance,shorten <=\shortendistance}}

\tikzset{lin shorten/.style={shorten >=\shortendistance,shorten <=\shortendistance,line width=1pt,decoration={markings, mark=at position 0.5*\pgfdecoratedpathlength+1.8pt with {\arrow{angle 90}}}, postaction={decorate}}}

\usetikzlibrary{calc,decorations.markings,decorations.pathmorphing,arrows,external,decorations.pathreplacing,arrows.meta,calc,decorations.markings,decorations.pathmorphing,shapes}

\newcommand{\vertexat}[1]{\fill [very thick,draw=white,fill=black] #1 circle (3pt);}

\definecolor{tnred}{RGB}{     255, 63,   63}
\definecolor{tnorange}{RGB}{  220, 140,   6}
\definecolor{tnblue}{RGB}{     10, 154, 215}
\definecolor{tngreen}{RGB}{   114, 165,  10}
\definecolor{tnmagenta}{RGB}{ 215,  10, 154}
\definecolor{tnbg}{RGB}{252,252,252}
\definecolor{tnfg}{RGB}{5,8,12}

\begin{tikzpicture}

    \node at (1,4.3) {$S$};
    \node at (3,4.3) {$N(S)$};

    \draw [dotted,thick,gray!50] (2,-1.5) -- (2,5);
    \draw [dotted,thick,gray!50] (4,-1.5) -- (4,5);

\def\midx{3}
    \coordinate (a) at ($(5,-1)$);
    \coordinate (b) at ($(\midx,-1)$);
    \coordinate (c) at ($(1, -1)$);
    \coordinate (d) at ($(1, 0)$);
    \coordinate (e) at ($(\midx, 0)$);
    \coordinate (f) at ($(5, .5)$);
    \coordinate (g) at ($(\midx, 1)$);
    \coordinate (h) at ($(1, 1)$);
    \coordinate (i) at ($(0,1.5)$);
    \coordinate (j) at ($(1,2)$);
    \coordinate (k) at ($(\midx,2)$);
    \coordinate (l) at ($(5,2.5)$);
    \coordinate (m) at ($(\midx,3)$);
    \coordinate (n) at ($(5,3.5)$);
    
    \draw [highlight2,rounded corners,line cap=round,opacity=0.3] (a) to (b) to (c) to (d) to (e) to (f) to(g) to (h) to (i) to (j) to (k) to (l) to (m) to (n);

    \draw [plainlin] (a) -- (b);
    \draw [dashdotted, very thick,red] (b) -- (c);
    \draw [dashdotted, very thick,red] (c) -- (d);
    \draw [dashdotted, very thick,red] (d) -- (e);
    \draw [plainlin]
    (d) -- (h);
    \draw [plainlin]
    (d) -- (i);
    \draw [plainlin]
    (d) -- (g);
    \draw [plainlin] (e) -- (f);
    \draw [plainlin] (f) -- (g);
    \draw [dashdotted, very thick,red] (g) -- (h);
    \draw [dashdotted, very thick,red] (h) -- (i);
    \draw [plainlin]
    (h) -- (j);
    \draw [dashdotted, very thick,red] (i) -- (j);
    \draw [dashdotted, very thick,red] (j) -- (k);
    \draw [plainlin]
    (j) -- (m);
    \draw [plainlin] (k) -- (l);
    \draw [plainlin] (l) -- (m);
    \draw [plainlin] (m) -- (n);

    \foreach \nn in {a,b,c,d,e,f,g,h,i,j,k,l,m,n} {
      \vertexat{(\nn)}
    }

\end{tikzpicture}
    \end{minipage}
    \hfill
    \begin{minipage}{.48\textwidth}
      \vspace*{1.5cm}
      
\definecolor{tnred}{RGB}{     255, 63,   63}
\definecolor{tnorange}{RGB}{  220, 140,   6}
\definecolor{tnblue}{RGB}{     10, 154, 215}
\definecolor{tngreen}{RGB}{   114, 165,  10}
\definecolor{tnmagenta}{RGB}{ 215,  10, 154}
\definecolor{tnbg}{RGB}{252,252,252}
\definecolor{tnfg}{RGB}{5,8,12}

\tikzset{highlight/.style={orange!75!white,line width=6pt,line cap=round}}
\tikzset{highlight2/.style={tngreen!70!yellow!60,line width=11pt,line cap=round}}
\tikzset{psline/.style={decorate,decoration={amplitude=0.3mm,segment length=2mm,post length=1.5mm,pre length=1.5mm,coil,aspect=0}}}
\tikzset{plainlin/.style={line width=1pt}}
\def\shortendistance{5pt}
\tikzset{plainlin shorten/.style={line width=1pt,shorten >=\shortendistance,shorten <=\shortendistance}}

\tikzset{lin shorten/.style={shorten >=\shortendistance,shorten <=\shortendistance,line width=1pt,decoration={markings, mark=at position 0.5*\pgfdecoratedpathlength+1.8pt with {\arrow{angle 90}}}, postaction={decorate}}}
\tikzsetnextfilename{fig-separation}

\newcommand{\vertexat}[1]{\fill [very thick,draw=white,fill=black] #1 circle (3pt);}

\begin{tikzpicture}
  \node[shape=circle,fill=gray!40,minimum size=1.4cm] (H')  at (-2, 0)    {$H'$};
  \node[shape=circle,fill=gray!40,minimum size=1.4cm] (G')  at ( 2, 0)    {$G'$};
  \node[shape=circle,fill=gray!40,minimum size=1.4cm] (PS) at ( 0,-3) {$K_S$};

   \begin{scope}[->, shorten >=2pt,thick]
     \draw (H') .. controls +(180:1.5cm) and +(110:1.5cm) .. node [midway,xshift=-.25cm,yshift=.25cm] {$\calk$} (H');

    \draw (G') .. controls +(80:1.5cm) and +(0:1.5cm) .. node [midway,xshift=.25cm,yshift=.25cm] {$\mathcal{S}$} (G');

    \draw (PS) .. controls +(-115:1.5cm) and +(-75:1.5cm) .. node [midway,xshift=0.1cm,yshift=-.3cm] {$\mathcal{K}_S$} (PS);

    \draw (H') .. controls +(40:1.5cm) and +(150:1.5cm) ..  node [pos=.3,yshift=-.3cm] {$\mathcal{D}'$}  (G');
    \draw (G') .. controls +(220:1.5cm) and +(-30:1.5cm) ..  node [pos=.3,yshift=.3cm] {$\mathcal{S}$}  (H');

    \draw (H') .. controls +(-110:2cm) and +(140:2cm) ..  node [pos=.45,yshift=.3cm,xshift=.2cm] {$\mathcal{D}'$}  (PS);
    \draw (PS) .. controls +(170:2.2cm) and +(-130:3cm) ..  node [pos=.2,yshift=-.3cm] {$\mathcal{K}_S$}  (H');

     \draw (G') .. controls +(-70:2cm) and +(40:2cm) ..  node [pos=.45,yshift=.3cm,xshift=-.2cm] {$\mathcal{S}$}  (PS);
     \draw (PS) .. controls +(10:2.2cm) and +(-50:3cm) ..  node [pos=.2,yshift=-.3cm] {$\mathcal{K}_S$}  (G');


 \end{scope}

\end{tikzpicture}
    \end{minipage}
    \vspace*{-1.3\baselineskip}
    \caption{Left: a set $S$ in a Hamiltonian graph and its neighborhood~$N(S)$;
              part of a Hamiltonian cycle is highlighted, dashed red edges are the edges in~$C_S$. Right: cycles and edges that separate subgraphs of~$G$, where $A\stackrel{\alpha}{\to}B$ indicates that $\alpha$ separates $A$ from~$B$ (for instance, $\mathcal{Q'}$ separates $G'$ from $H'$).}\label{fig:NS}
  \end{center}
\end{figure}
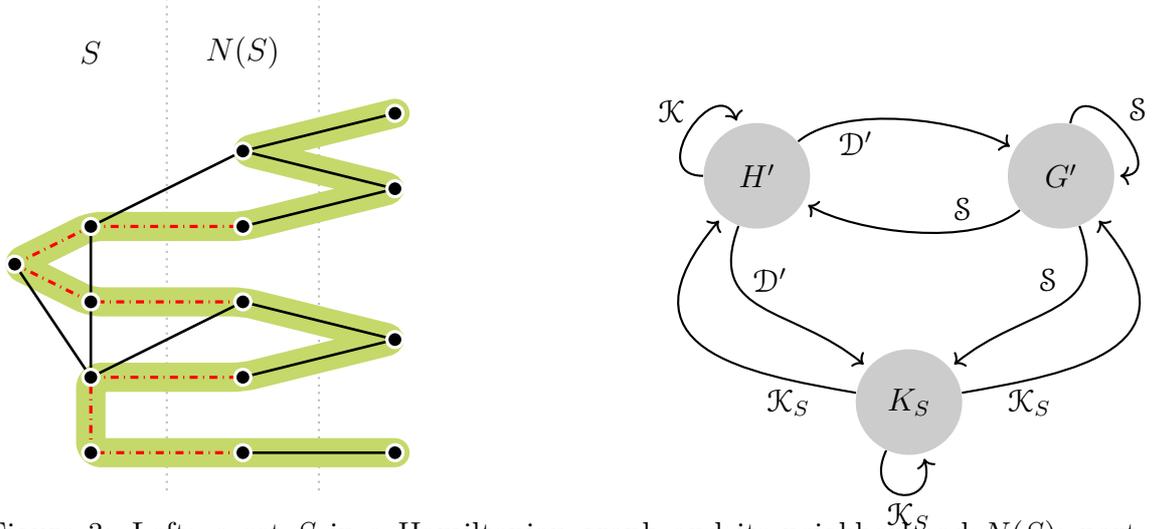

\medskip\noindent
It remains to create a set of at most \(24h \leq 72|S|\) edges and subdivisions of \(K_4\)
that separates \(H'\) from itself  (see Figure~\ref{fig:NS}).
In order to obtain the final elements of the separating system,
we define special matchings.

\medskip
\noindent
\textbf{Easy matchings.}
A \defi{slicing} $A_1,\dots,A_s$ of $[a_1,h]$ is a sequence of intervals which together partition~$[a_1,h]$.
We always assume that intervals are  ordered consecutively, i.e.,
so that each $x\in A_i$ precedes all $y\in \bigcup_{j>i} A_j$.
Given an interval $A\subseteq [h]$,
we say that \(v_iv_j\) \emph{starts} in \(A\) if \(i\in A\);
and that \(v_iv_j\) \emph{ends} in \(A\) if \(j\in A\).
We say that a matching \(M\subseteq E(H')\) is \emph{easy}
if there is \(a_1\in [h]\),
a slicing \(A_1,\ldots, A_s\) of \([a_1,h]\)
and \(t\in\{0,1\}\) such that the following hold.
\begin{enumerate}[label=\textit{(\alph*)},ref=\textit{(\alph*)}]
\item\label{i:alternating}
  for each $v_iv_j\in M$ there is $r\in[s-1]$ such that $r\equiv t\pmod 2$, $v_i\in A_r$ and $j\in A_{r+1}$;
\item\label{i:crossing}
 the edges starting in $A_r$ are pairwise crossing, for all $r\in[s]$; and
\item\label{i:zero-or-odd}
  for each $r\in [s]$, the number of edges starting in $A_r$ is either zero or is odd.
\end{enumerate}

As we shall see, \ref{i:alternating}--\ref{i:zero-or-odd} suffice to overcome
our main obstruction when covering matchings by $K_4$ subdivisions.

\bigskip
\noindent
Now, for \(k \in [h-1]\),
let \(M_k = \{v_iv_j \in E(H') : j-i = k\}\),
and for \(k \in [-h+2,h-2]\),
let \(N_k = \{v_iv_j \in E(H') : j-2i = k\}\).
It is not hard to check that \(M_k\) and \(N_k\) are linear forests,
and each edge of \(H'\) is in precisely one \(M_k\) and precisely one \(N_k\).
This defines at most \(3h\) linear forests.
In what follows, we partition each \(M_k\) and \(N_k\) into four easy matchings,
yielding \(12h\) easy matchings.
First, in items (i) and (i') below,
we partition, respectively, each \(M_k\) and \(N_k\) into two matchings satisfying \ref{i:alternating} and~\ref{i:crossing}.

\begin{enumerate}
\item[\refstepcounter{enumi}\textit{(i)}]\label{i:(i)}
Let $k\in[h-1]$.
We partition \(M_k\) as follows.
Let \(A_1,\ldots,A_{\lceil n/k\rceil}\) be a slicing of \([h]\)
into sets of size \(k\) and at most one set \(A_{\lceil n/k\rceil}\) of size at most~\(k\).
Now, given \(t \in \{0,1\}\),
we let \(\mathcal{A}_t = \bigcup_{r \equiv t \pmod{2}} A_r\),
and then put \(M_{k,t} = \{v_iv_j \in M_k : i \in \mathcal{A}_t\}\).
So, for example, \(M_{k,1}\) consists of the edges of \(M_k\) that start in ``odd'' intervals of~\([h]\).
If \(i\in A_r\), then \(j\in A_{r+1}\) because \(j-i = k=|A_r|\) (see Figure~\ref{fig:step2}).
This implies that if \(v_iv_j\in M_{k,t}\), then \(i\in \mathcal{A}_t\) while \(j\in\mathcal{A}_{1-t}\).
In particular, \(M_{k,t}\) is a matching.
Therefore, the matchings $M_{k,t}$ satisfy~\ref{i:alternating}.
Now, suppose that \(v_iv_j,v_{i'}v_{j'}\in M_{k,t}\) are such that \(i,i'\in A_r\)
for some \(r\), and suppose, without loss of generality, that \(i < i'\).
Then \(j = k+i < k+i' = j'\), and hence \(v_iv_j\) and \(v_{i'}v_{j'}\) cross.
Therefore, the matchings \(M_{k,t}\) satisfy~\ref{i:crossing} as desired.
This step defined at most \(2h\) matchings.

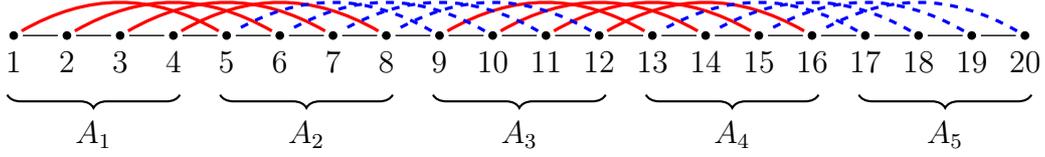
\begin{figure}[h]
\centering
\begin{tikzpicture}[node distance=2cm,scale = .7]
\tikzset{black vertex/.style={circle,draw,minimum size=1mm,inner sep=0pt,outer sep=2pt,fill=black, color=black}}

\foreach \x in {1,...,20}{
    \node[black vertex] (\x) at (\x,0) {};
	\node () at (\x,-.5) {\x};
}

\draw[line width = .5pt] (1) -- (2) -- (3) -- (4) -- (5) -- (6) -- (7) -- (8) -- (9) -- (10) -- (11) -- (12) -- (13) -- (14) -- (15) -- (16) -- (17) -- (18) -- (19) -- (20);

\foreach \i in {5,6,7,8,13,14,15,16}
{
    \pgfmathtruncatemacro{\j}{(\i-4)}
    \ifnum\j>0
        \draw[bend right,color=red,line width = 1.2pt] (\i) to (\j);
    \fi
}

\foreach \i in {9,10,11,12,17,18,19,20}
{
    \pgfmathtruncatemacro{\j}{(\i-4)}
    \ifnum\j>0
        \draw[bend right,color=blue,line width = 1.2pt, dashed] (\i) to (\j);
    \fi
}

\draw [thick,decoration={brace,mirror,raise=0.7cm, amplitude=5pt},decorate] (1.south west) -- (4.south east) node [pos=0.5,anchor=north,yshift=-0.9cm] {$A_1$};
\draw [thick,decoration={brace,mirror,raise=0.7cm, amplitude=5pt},decorate] (5.south west) -- (8.south east) node [pos=0.5,anchor=north,yshift=-0.9cm] {$A_2$};
\draw [thick,decoration={brace,mirror,raise=0.7cm, amplitude=5pt},decorate] (9.south west) -- (12.south east) node [pos=0.5,anchor=north,yshift=-0.9cm] {$A_3$};
\draw [thick,decoration={brace,mirror,raise=0.7cm, amplitude=5pt},decorate] (13.south west) -- (16.south east) node [pos=0.5,anchor=north,yshift=-0.9cm] {$A_4$};
\draw [thick,decoration={brace,mirror,raise=0.7cm, amplitude=5pt},decorate] (17.south west) -- (20.south east) node [pos=0.5,anchor=north,yshift=-0.9cm] {$A_5$};

\end{tikzpicture}
\caption{A partition of \(M_{4}\) into \(M_{4,0}\) and \(M_{4,1}\) in, respectively, dashed blue and solid red.}
\label{fig:step2}
\end{figure}

\item[\refstepcounter{enumi}\textit{(i')}]\label{i:(i')}
Let \(k\in[-h+2,h-2]\).
We partition \(N_k\) as follows.
We consider two cases. If $k \ge 0$,
then let $a_1=1$; otherwise let $a_1=1-k\ge 2$.
Now, for each $r \ge 1$ let $a_{r+1} = 2a_{r} +k$.
Since $a_1> -k$, we can prove by induction that $a_{r+1}>a_r> -k$ for all $r$.
Also, let \(d\) be the minimum integer for which \(h < a_{d+1}\).
We partition \([a_1,h]\) into sets \(A_1,\ldots, A_d\)
so that \(A_r = [a_r,a_{r+1}-1]\) for \(r\in [d-1]\),
and \(A_d = [a_d,h]\).

Note that every edge of $N_k$ starts in $\bigcup_{j\in[d]}A_d$.
Indeed, let \(v_iv_j\in N_k\).
Then \(j=k +2i\) and we claim that \(i \geq a_1\).
This is clear if \(a_1 = 1\).
Otherwise, \(a_1 = 1-k\),
but since \(j > i\), we must have \(k + 2i \geq i+1\), and hence \(i \geq 1-k = a_1\).
Moreover, if \(v_iv_j\) starts in \(A_r\),
then \(v_iv_j\) ends in $A_{r+1}$ since
\[
  a_{r+1}
  =     k + 2a_r
  \le   k + 2i
  \le   k + 2(a_{r+1}-1)
  <     k + 2a_{r+1}
  =     a_{r+2}.
\]

Given \(t \in \{0,1\}\),
we let \(\mathcal{A}_t = \bigcup_{r \equiv t \pmod{2}} A_r\),
and put \(N_{k,t} = \{v_iv_j \in N_k :  i \in \mathcal{A}_t\}\).
As in the preceding case, we can prove that
the \(N_{k,t}\) are matchings that satisfy
both \ref{i:alternating}~and~\ref{i:crossing} as desired.
This step defined at most \(4h\) matchings.

\begin{figure}[h]
\centering
\begin{tikzpicture}[node distance=2cm,scale = .5]
\tikzset{black vertex/.style={circle,draw,minimum size=1mm,inner sep=0pt,outer sep=2pt,fill=black, color=black}}

\foreach \x in {1,...,30}{
    \node[black vertex] (\x) at (\x,0) {};
	\node () at (\x,-.5) {\x};
}

\draw[line width = .5pt] (1) -- (2) -- (3) -- (4) -- (5) -- (6) -- (7) -- (8) -- (9) -- (10) -- (11) -- (12) -- (13) -- (14) -- (15) -- (16) -- (17) -- (18) -- (19) -- (20) -- (21) -- (22) -- (23) -- (24) -- (25) -- (26) -- (27) -- (28) -- (29) -- (30);

\foreach \i in {3,4,5,6}
{
    \pgfmathtruncatemacro{\j}{(2*\i+1))}
    \ifnum\j>0
        \draw[bend left,color=blue,line width = 1.2pt, dashed] (\i) to (\j);
    \fi
}

\foreach \i in {1,2,7,8,9,10,11,12,13,14}
{
    \pgfmathtruncatemacro{\j}{(2*\i+1))}
    \ifnum\j>0
        \draw[bend left,color=red,line width = 1.2pt] (\i) to (\j);
    \fi
}

\draw [thick,decoration={brace,mirror,raise=0.7cm, amplitude=5pt},decorate] (1.south west) -- (2.south east) node [pos=0.5,anchor=north,yshift=-0.9cm] {$A_1$};
\draw [thick,decoration={brace,mirror,raise=0.7cm, amplitude=5pt},decorate] (3.south west) -- (6.south east) node [pos=0.5,anchor=north,yshift=-0.9cm] {$A_2$};
\draw [thick,decoration={brace,mirror,raise=0.7cm, amplitude=5pt},decorate] (7.south west) -- (14.south east) node [pos=0.5,anchor=north,yshift=-0.9cm] {$A_3$};
\draw [thick,decoration={brace,mirror,raise=0.7cm, amplitude=5pt},decorate] (15.south west) -- (30.south east) node [pos=0.5,anchor=north,yshift=-0.9cm] {$A_4$};

\end{tikzpicture}
\caption{A partition of \(N_{1}\) into \(N_{1,0}\) and \(N_{1,1}\) in, respectively, dashed blue and solid red.}
\label{fig:step2'}
\end{figure}
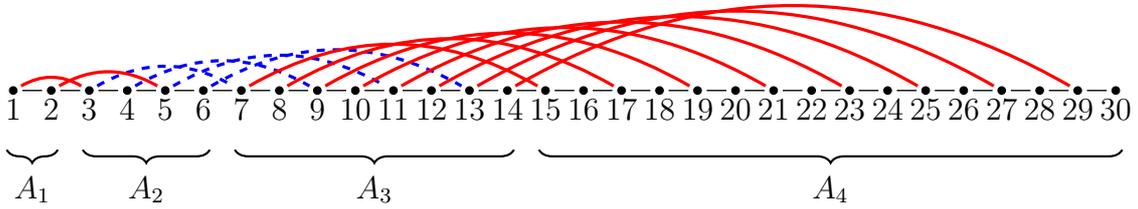

\item[\refstepcounter{enumi}\textit{(ii)}]\label{i:(ii)}
Finally, we partition each \(M_{k,t}\) and \(N_{k,t}\) obtained in steps \hyperref[i:(i)]{\textit{(i)}} and~\hyperref[i:(i')]{\textit{(i')}}
into matchings satisfying \ref{i:zero-or-odd}.
It is not hard to see that properties \ref{i:alternating}~and~\ref{i:crossing} are hereditary,
i.e., that if a matching $M$ satisfies \ref{i:alternating}~and~\ref{i:crossing}, then any subset of $M$ satisfies \ref{i:alternating}~and~\ref{i:crossing}.
Now, for each \(k\in [h]\) and each \(t\in\{0,1\}\)
we partition \(M_{k,t}\) into sets \(M_{k,t,0}\) and \(M_{k,t,1}\)
so that \(M_{k,t,0}\) and \(M_{k,t,1}\) have either zero or an odd number of edges starting in each~\(A_r\).
This works because every even-sized matching can be partitioned into two odd-sized matchings
and every odd-sized matching can be partitioned into an odd-sized matching and an empty matching.
We obtain \(N_{k,t,0}\) and~\(N_{k,t,1}\) analogously.
This step defined at most~\(12h\) matchings.

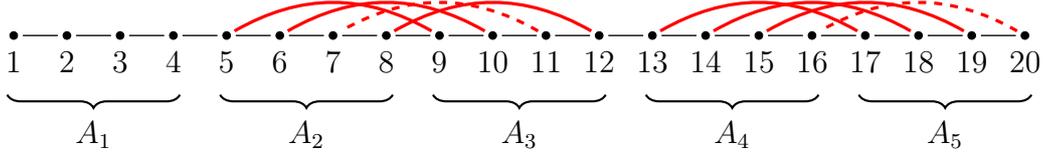
\begin{figure}[h]
\centering
\begin{tikzpicture}[node distance=2cm,scale = .7]
\tikzset{black vertex/.style={circle,draw,minimum size=1mm,inner sep=0pt,outer sep=2pt,fill=black, color=black}}

\foreach \x in {1,...,20}{
    \node[black vertex] (\x) at (\x,0) {};
	\node () at (\x,-.5) {\x};
}

\draw[line width = .5pt] (1) -- (2) -- (3) -- (4) -- (5) -- (6) -- (7) -- (8) -- (9) -- (10) -- (11) -- (12) -- (13) -- (14) -- (15) -- (16) -- (17) -- (18) -- (19) -- (20);


\foreach \i in {9,10,12,17,18,19}
{
    \pgfmathtruncatemacro{\j}{(\i-4)}
    \ifnum\j>0
        \draw[bend right,color=red,line width = 1.2pt] (\i) to (\j);
    \fi
}

\foreach \i in {11,20}
{
    \pgfmathtruncatemacro{\j}{(\i-4)}
    \ifnum\j>0
        \draw[bend right,color=red,line width = 1.2pt,dashed] (\i) to (\j);
    \fi
}

\draw [thick,decoration={brace,mirror,raise=0.7cm, amplitude=5pt},decorate] (1.south west) -- (4.south east) node [pos=0.5,anchor=north,yshift=-0.9cm] {$A_1$};
\draw [thick,decoration={brace,mirror,raise=0.7cm, amplitude=5pt},decorate] (5.south west) -- (8.south east) node [pos=0.5,anchor=north,yshift=-0.9cm] {$A_2$};
\draw [thick,decoration={brace,mirror,raise=0.7cm, amplitude=5pt},decorate] (9.south west) -- (12.south east) node [pos=0.5,anchor=north,yshift=-0.9cm] {$A_3$};
\draw [thick,decoration={brace,mirror,raise=0.7cm, amplitude=5pt},decorate] (13.south west) -- (16.south east) node [pos=0.5,anchor=north,yshift=-0.9cm] {$A_4$};
\draw [thick,decoration={brace,mirror,raise=0.7cm, amplitude=5pt},decorate] (17.south west) -- (20.south east) node [pos=0.5,anchor=north,yshift=-0.9cm] {$A_5$};

\end{tikzpicture}
\caption{A partition of \(M_{4,0}\) into \(M'_{4,0}\) and \(M''_{4,0}\) in, respectively, solid red and dashed red.}
\label{fig:step3}
\end{figure}
\end{enumerate}

Let
\begin{align*}
  \mathcal{M} &\quad =\quad \bigl\{M_{k,t,\rho} : k\in[h],\ t, \rho\in\{0,1\}\bigr\} \\
  \mathcal{N} &\quad =\quad \bigl\{N_{k,t,\rho} : k\in[-h+2,h-2],\ t, \rho\in\{0,1\}\bigr\}
\end{align*}
and note that
\[
  \bigcup_{M\in\mathcal{M}} M
  =  E(H')
  = \bigcup_{M\in\mathcal{N}} M.
\]

\smallskip
\noindent
\textbf{Building subdivisions of \(K_4\).}
In what follows, given \(M\in \mathcal{M}\cup\mathcal{N}\) we show that \(K\cup M\) contains either one subdivision of \(K_4\) that contains~\(M\);
or one edge and a subdivision of $K_4$ that together cover~$M$.
Let $A_1,\dots,A_d$ be the slicing defined in \hyperref[i:(i)]{\textit{(i)}} or \hyperref[i:(i')]{\textit{(i')}} when selecting edges for~$M$.
This construction is divided into two cases.
We say that \(M\) is \emph{elementary} if for every~\(r\in[d]\),
there is at most one edge of \(M\) that starts in~\(A_r\);
otherwise we say that \(M\) is \emph{jumbled}.
By construction, if \(M\) is elementary, then no two edges of \(M\)~cross.
Recall that \(e\) and \(e'\) are edges not in \(H'\) that we use to complete \(C\) to a subdivision \(K\) of~\(K_4\).

\medskip
\noindent
\textbf{Elementary.}
Suppose \(M\) is elementary.
We divide this proof into two cases,
depending on whether \(M\) contains an edge crossing \(e\) or \(e'\).
Let \(w_1,w_2,w_3,w_4\) be the vertices of \(e\) and \(e'\) in the order that they appear in \(C\),
and assume, without loss of generality, that \(e = w_1w_3\) and \(e' = w_2w_4\).

\smallskip
\noindent
\underline{Case 1.}
Suppose that no edge of \(M\) crosses \(e\) or \(e'\).
Since no two edges of \(M\) cross, there is at most one edge, say \(f = v_iv_j\), of \(M\)
that starts before \(w_1\) and ends after \(w_4\),
in which case we pick $f$ as a single edge.
In either case, we cover $M$ (or $M\setminus\{f\}$) by a subdivision~$K_M$ of~$K_4$
obtained from $K$ by ``taking shortcuts'':
For every edge \(v_iv_j \in M\) (or~\(v_iv_j \in M\setminus\{f\}\)),
we replace the path \(C[v_i,v_j]\) by the edge~\(v_iv_j\).

\smallskip
\noindent
\underline{Case 2.}
Suppose that \(M\) contains some edge \(f\) that crosses \(e\) or \(e'\).
Assume without loss of generality that $f$ and~\(e\) cross.
Note, in particular, that by the construction of $M$, no edge of $M$ can start before $w_1$ and also end after $w_4$.
Since no two edges of \(M\) cross, there is at most one other edge, say \(f'\), of \(M\) that crosses~\(e\),
in which case we pick $f'$ as a single edge.
Then \(K' = C + f + e\) is a subdivision of \(K_4\), and we obtain a subdivision \(K_M\) of \(K_4\) from \(K'\)
analogously to Case 1: for every edge \(v_iv_j \in M\setminus\{f\}\) (or \(v_iv_j \in M\setminus\{f,f'\}\)),
we replace the path \(C[v_i,v_j]\) by \(v_iv_j\).

This concludes the analysis when \(M\) is elementary.

\medskip
\noindent
\textbf{Jumbled.}
In this case we do not use \(e\) and \(e'\).
The idea here is to modify \(C\) into a cycle \(C'\subseteq C \cup M\) that contains \(M\),
and then restore some of the edges removed from \(C\) to obtain a subdivision of \(K_4\).
Recall that $M$ is easy, and hence there exist  \(a_1\), \(A_1,\ldots,A_d\), and $t$
such that~\ref{i:alternating}--\ref{i:zero-or-odd} hold.

\smallskip
\noindent
\underline{Modifying \(C\).}
Let \(A_r\) be an interval of \([h]\) for which the set \(M_r\) of edges of \(M\) that start in \(A_r\) (and end in \(A_{r+1}\)) is non-empty.
Then, by \ref{i:alternating}, no edge of \(M\) starts in \(A_{r+1}\).
Let \(Q_r\) be the shortest subpath of \(C\) that contains the vertices \(v_i\) with \(i\in A_r\cup A_{r+1}\),
and let \(R_r = Q_r\cup M_r\).
We show that there is a path \(Q_r'\subseteq R_r\) that contains \(M_r\)
and has the same end vertices as \(Q_r\).
Then \(C'\) is obtained from $C$ by replacing each~$Q_r$ by $Q_r'$.

The existence of $Q_r'$ can easily be shown by induction.
Alternatively, one can use the following construction.
Let \(s = |M_r|\) and let \(w_1, \ldots, w_{2s}\) be the vertices of \(Q_r\)
incident to edges of \(M_r\) in the order that they appear in \(Q_r\).
Note that \(w_1\) and \(w_{2s}\) are the end vertices of \(Q_r\),
and that $s$ is~odd by~\ref{i:zero-or-odd}.
By~\ref{i:crossing}, \(M_r = \{w_iw_{i+s} : i\leq s\}\).
Now, let \(Q_r'\) be the graph obtained from \(R_r\)
by removing the edges in \(Q_r[w_i,w_{i+1}]\) for odd~\(i\) with \(i \leq 2s-1\)
and removing isolated vertices (see Figure~\ref{fig:partial-path}).
Note that \(R_r\) is a subcubic graph whose vertices with degree~\(3\)
are precisely \(w_2, \ldots, w_{2s-1}\),
and that, to obtain \(Q_r'\), we remove from~\(R_r\) precisely one~edge incident to each \(w_i\).
Thus \(\Delta(Q_r') \leq 2\).
  We claim that $Q_r'$ is connected.
  For that, we prove that $Q_r'$ contains a path joining $w_i$~to~$w_{i+1}$ for each $i\in[2s-1]$.
  This is clear if $i$ is even, because the segment $Q_r[w_{i}w_{i+1}]$ has not been removed.
  For odd~$i$, on the other hand,
  the segment~$Q_r[w_{i+s}w_{i+s+1}]$ has not been removed because $s$ is~odd (and consequently~$i+s$ is even), and
  thus $w_iw_{i+s}Q_r[w_{i+s}w_{i+s+1}]w_{i+1}$ is the desired a path in~$Q_r'$.
Moreover, the end vertices of \(Q_r'\) are \(w_1\) and \(w_{2s}\) and  \(M_r\subseteq E(Q_r')\).
This concludes the construction of~$Q_r'$.

  \begin{figure}[h]
\centering
\begin{tikzpicture}[node distance=2cm,scale = .8]
\tikzset{black vertex/.style={circle,draw,minimum size=1mm,inner sep=0pt,outer sep=2pt,fill=black, color=black}}

\foreach \x in {1,...,18}{
    \node[black vertex] (\x) at (\x,0) {};
	\node () at (\x,-.5) {\x};
}

\node () at (1,-1) {$w_1$};
\node () at (2,-1) {$w_2$};
\node () at (3,-1) {$w_3$};
\node () at (5,-1) {$w_4$};
\node () at (6,-1) {$w_5$};

\node () at (8,-1) {$w_6$};
\node () at (10,-1) {$w_7$};
\node () at (12,-1) {$w_8$};
\node () at (16,-1) {$w_9$};
\node () at (18,-1) {$w_{10}$};


\draw[line width = 2.5pt] (1)  (2) -- (3) (5) -- (6) (8) -- (9) -- (10) (12) -- (13) -- (14) -- (15) -- (16)  (18);

\draw[line width = .5pt] (1) -- (2) -- (3) -- (4) -- (5) -- (6) -- (7) -- (8) -- (9) -- (10) -- (11) -- (12) -- (13) -- (14) -- (15) -- (16) -- (17) -- (18);

\foreach \i in {8,10,12,16,18}
{
    \pgfmathtruncatemacro{\j}{(\i-6)/2}
    \ifnum\j>0
        \draw[bend right,line width = 2.5pt] (\i) to (\j);
    \fi
}

\end{tikzpicture}
\caption{Thick lines illustrate the path \(Q_r'\) that contains \(M_r\subseteq N_{6,t,r}\) for some \(t\) and \(r\).}
\label{fig:partial-path}
\end{figure}
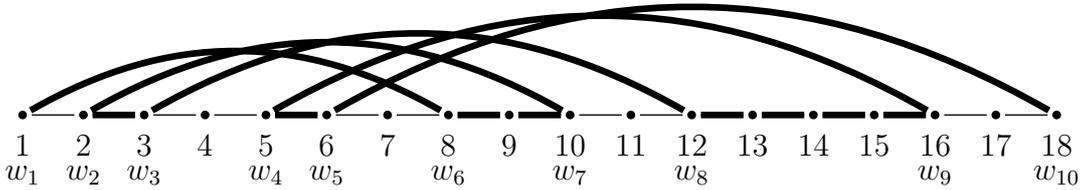

\smallskip
\noindent
\underline{Turning \(C'\) into a subdivision of \(K_4\).}
To turn \(C\) into a subdivision of \(K_4\) we only need to restore two of the removed \(Q_r[w_i,w_{i+1}]\) for some~$r$.
More precisely, since \(M\) is jumbled,
there is an interval \(A_r\) in which at least two edges of \(M\) start.
Let \(M_r\subseteq M\) be the set of edges that start in \(A_r\).
Let \(s = |M_r|\) and let \(w_1, \ldots, w_{2s}\) be the vertices of \(P\)
incident to edges of \(M_r\) as above.
Recall that the removed subpaths were \(Q_r[w_i,w_{i+1}]\) for odd~\(i\).
Moreover, recall that \(s\) is odd by~\ref{i:zero-or-odd} and thus~$s\ge 3$.
Therefore, \(Q_r[w_1,w_2]\) and \(Q_r[w_{s+2},w_{s+3}]\) were removed.
Let \(K_M = C'\cup Q_r[w_1,w_2] \cup Q_r[w_{s+2},w_{s+3}]\).
We claim that \(K_M\) is a subdivision of~\(K_4\).
Indeed, by construction of \(C'\), the order in which the vertices \(w_1,\ldots,w_{2s}\) appear
is
\[
  w_1,w_{s+1},w_{s+2},w_2,w_3,w_{s+3}, \ldots ,w_i,w_{i+s},w_{i+s+1},w_{i+1}, \ldots, w_{2s}.
\]
Let \(p(w_i)\) be the position of \(w_i\) in the order above.
In particular, we have \(p(w_1) = 1\), \(p(w_{s+2}) = 3\), \(p(w_{2}) = 4\) and \(p(w_{s+3}) = 6\).
Then, \(p(w_1) < p(w_{s+2}) < p(w_2) < p(w_{s+3})\), and hence
\(K_M\) is a subdivision of \(K_4\).
This step creates a single subdivision of~\(K_4\) for each \(M\).

\smallskip
Now, by applying the construction above on each \(M\) in \(\mathcal{M}\cup\mathcal{N}\),
we obtain a set \(\mathcal{K}\) of at most \(24h\) subdivisions of \(K_4\).
We claim that \(\calk\) separates each pair of edges in~\(H'\).
Indeed, let $v_iv_j$ and $v_{i'}v_{j'}$ be edges of~\(H'\).
Each such edge belongs to exactly one $M\in\mathcal{M}$ and exactly one $N\in\mathcal{N}$.
If they belong to different elements of \(\mathcal{M}\),
then they belong to different elements of~\(\mathcal{K}\) and are separated,
because $\mathcal{M}$ partitions~$E(H')$.
The same argument works for~\(\mathcal{N}\).
Therefore, we may assume that $i-j=i'-j'$ and $i-2 j=i'-2 j'$.
This immediately yields $j=j'$ and $i=i'$, so $v_iv_j = v_{i'}v_{j'}$.
Therefore, any two \emph{distinct} edges in $H'$ are separated by~$\mathcal{K}$, as desired.
This concludes the proof.
\end{proof}

\section{Separating into cycles}
\label{sec:cycles}

First, note that \(K_4\) can be covered by two cycles
(indeed, suppose \(V(K_4) = [4]\) and note that the cycles \(1 2 3 4 1\) and \(1 2 4 3 1\) cover it).
Hence every subdivision of~\(K_4\) can be covered by two cycles.
This means that any \(K_4\)-separating system~$\cals$
  yields a \(K_3\)-separating system of size at most \(2|\cals|\).
  And thus, by Theorem~\ref{thm:K4-separation},
  every graph on \(n\) vertices admits a \(K_3\)-separating system of size at most \(164n\),
  which verifies Conjecture~\ref{conj:edge-separation} when \(H = K_3\).
A slightly better result, Theorem~\ref{thm:cycles} (any graph on \(n\) vertices
can be separated by at most \(41 n\) edges and cycles) can be obtained
using the following theorem of Pyber~\cite{pyber}.

\begin{theorem}[\cite{pyber}]\label{thm:pyber}
  Every graph~$G$ contains \(|V(G)|-1\) cycles and edges covering~$E(G)$.
\end{theorem}

The proof Theorem~\ref{thm:cycles} follows that of Theorem~\ref{thm:K4-separation}.
However, after partitioning the \(M_k\)'s and \(N_k\)'s into \(12h\) special matchings,
there is no need to consider elementary and jumbled cases,
since there is no need to find crossing edges (one proceed directly
to the ``Modifying \(C\)'' part of the argument,
in which we obtain a cycle \(C'\)).
Of course, one should also skip the step for ``Turning \(C'\) into a subdivision of \(K_4\)''.

\section{Conclusion}

In this paper we give the first steps towards showing that every graph admits
a linearly-sized separating system formed by edges and subdivisions of $K_t$, for fixed~$t$.
Perhaps these techniques can be fitted to find subdivisions of larger cliques.
If a suitable subdivision of~\(K_t\) ($t>4$) --- namely, a subdivision containing
a Hamilton cycle --- can be shown to exist for larger cliques,
then the tricks used in our proof might still be useful to produce a
$K_t$-separating system.

Conjecture~\ref{conj:edge-separation} could be weakened by allowing
a wider class of structures.
For example, one could seek a separating system
consisting of graphs that yield some fixed graph~$H$
through a series of edge contractions;
or consisting of immersions of~$H$ (for a more about immersions, see, e.g.,~\cite{meyniel1988problem}).

\section*{Acknowledgments}

We thank Marthe Bonamy for fruitful discussions and suggestions.

\medskip

\noindent
This research has been partially supported by Coordena\c cão de Aperfei\c coamento
de Pessoal de N\'\i vel Superior -- Brasil -- CAPES -- Finance Code 001.
F.~Botler is supported by CNPq {(\small 304315/2022-2)} and
CAPES {(\small 88887.878880/2023-00)}.
T Naia was supported by the Grant PID2020-113082GB-I00
funded by MICIU/AEI/10.13039/501100011033 and by Spanish
State Research Agency, through the Severo Ochoa and Mar\'ia de
Maeztu Program for Centers and Units of Excellence in R\&D
(CEX2020-001084-M).
CNPq is the National Council for Scientific and Technological Development of Brazil.

\bibliographystyle{amsplain}

{\small
  \vspace{-12pt}
\providecommand{\bysame}{\leavevmode\hbox to3em{\hrulefill}\thinspace}
\providecommand{\MR}{\relax\ifhmode\unskip\space\fi MR }
\providecommand{\MRhref}[2]{%
  \href{http://www.ams.org/mathscinet-getitem?mr=#1}{#2}
}
\providecommand{\href}[2]{#2}

}
\end{document}